\documentclass[12pt]{article} 
\usepackage{amsmath,amssymb,latexsym,amsthm,enumerate,amscd,hyperref} \usepackage[abbrev]{amsrefs} 
\usepackage{graphicx}
\usepackage{pinlabel}
\newtheoremstyle{dotless}{}{}{\itshape}{}{\bfseries}{}{ }{}

\newtheorem{Theorem}{Theorem}[section]

\newtheorem{corollary}[Theorem]{Corollary} 
 
\newtheorem{lemma}[Theorem]{Lemma}

\newtheorem*{mainthm}{Main Theorem}

\theoremstyle{definition} 
\newtheorem{definition}[Theorem]{Definition} 
\newtheorem{remark}[Theorem]{Remark}
\newtheorem{example}[Theorem]{Example}

\newtheorem*{giang}{Comments on Theorem 5.3}

\newcommand{\ca}{\mathcal {A}} 
\newcommand{\cb}{\mathcal {B}} 
\newcommand{\cac}{\mathcal {C}}

\newcommand{\ch}{\mathcal {H}}

\newcommand{\cs}{\mathcal {S}}

\newcommand{\cc}{{\mathbb C}} 
\newcommand{\ee}{{\mathbb E}}

\newcommand{\rr}{{\mathbb R}} 
\newcommand{\sphere}{{\mathbb S}} 

\newcommand{\zz}{{\mathbb Z}}

\newcommand{\minus}{{-1}}

\newcommand{\eltwo}{{\ell^2}}

\newcommand{\bD}{{\mathbf D}}

\newcommand{\cd}{\operatorname{cd}} 
\newcommand{\gdim}{\operatorname{gd}}
\newcommand{\obdim}{\operatorname{obdim}}
\newcommand{\actdim}{\operatorname{actdim}}
\newcommand{\vk}{\operatorname{vk}}
\newcommand{\vktwo}{\vk_{\zz/2}}

\newcommand{\raag}{\mathrm {RAAG}}

\def\clap#1{\hbox to 0pt{\hss#1\hss}}

\newcommand{\comment}[1]{} 
 
\newcommand{\ga}{\alpha} 
\newcommand{\gb}{\beta} 
\newcommand{\gd}{\delta} 

\newcommand{\geps}{\varepsilon} 
 
\newcommand{\gi}{\iota} 
 
\newcommand{\gs}{\sigma} 
 
\newcommand{\gt}{\tau}

\newcommand{\gD}{\Delta} 
\newcommand{\gO}{\Omega}

\newcommand{\ab}{{ab}}
\newcommand{\pat}{{(PA_T)^\ab }}

\newcommand{\vertex}{\operatorname{Vert}}

\newcommand{\Card}{\operatorname{Card}}

\newcommand{\id}{\operatorname{id}}

\newcommand{\Ima}{\operatorname{Im}}

\newcommand{\sal}{\operatorname{Sal}}

\newcommand{\cat}{\operatorname{CAT}} 
\newcommand{\cone}{\operatorname{Cone}}

\newcommand{\OL}{\operatorname{OL}}

	{ 
\end{enumerate}
} 
\newenvironment{enumerate1'}{ 
\begin{enumerate}[\upshape (1)$'$]}
	{ 
\end{enumerate}
} 
	{ 
\end{enumerate}
}

\newenvironment{enumeratea'}{ 
\begin{enumerate}[\upshape (a)$'$]}{ 
\end{enumerate}
} 

\numberwithin{equation}{section} 
\begin{document}
\title{Determining the action dimension of an Artin group by using its complex of abelian subgroups}

\author{Michael W. Davis 
\thanks{Partially supported by an NSA grant.} 
\and {Jingyin Huang}   
}
\date{\today} \maketitle

\begin{abstract}  \smallskip
Suppose that $(W,S)$ is a Coxeter system with associated Artin group $A$ and with a simplicial complex $L$ as its nerve.  We define the notion of a ``standard abelian subgroup'' in $A$.  The poset of  such subgroups in $A$ is parameterized by the poset of simplices in a certain subdivision $L_\oslash$ of  $L$. This complex of standard abelian subgroups is used to generalize an earlier result from the case of  right-angled Artin groups to case of general Artin groups, by calculating,  in many instances, the smallest dimension of a manifold model for $BA$. (This is the ``action dimension'' of $A$ denoted  $\actdim A$.)  If $H_d(L; \zz/2)\neq 0$, where $d=\dim L$, then $\actdim A \ge 2d+2$.  Moreover, when the $K(\pi,1)$-Conjecture holds for $A$, the inequality is an equality.
\smallskip 

\noindent
\textbf{AMS classification numbers}.  Primary: 20F36, 20F55, 20F65, 57S30, 57Q35, 
Secondary: 20J06, 32S22
\smallskip

\end{abstract}

\section*{Introduction} \label{s:intro}
Given a discrete, torsion-free group $\pi$, its \emph{geometric dimension}, $\gdim \pi$,  is the smallest dimension of a model  for its classifying space $B\pi$ by a CW complex; its \emph{action dimension},  $\actdim \pi$, is the smallest dimension of a manifold model for $B\pi$. From general principles, $\actdim \pi\le 2\gdim \pi$.  Taking the universal cover of such a manifold model, we see that $\actdim \pi$ can alternately be defined as the smallest dimension of  a contractible manifold which admits a proper $\pi$-action.  The basic method for calculating $\actdim \pi$ comes from work  of Bestvina, Kapovich and Kleiner \cite{bkk}.  They show that if there is a finite complex $K$ such that the $0$-skeleton of the cone (of infinite radius) on $K$ coarsely embeds in $\pi$ and if the van Kampen obstruction for embedding $K$ in $S^m$ is nonzero, then $\actdim \pi \ge m+2$  (such a complex $K$ is an \emph{$m$-obstructor}). 

A \emph{Coxeter matrix} $(m_{st})$ on a set $S$ is an $(S\times S)$-symmetric matrix with $1$\,s on the diagonal and with each off-diagonal entry either an integer $\ge 2$ or the symbol $\infty$.  Associated to a Coxeter matrix there  is a \emph{Coxeter group} $W$ defined by the presentation with set of generators $S$ and with relations: $(st)^{m_{st}}=1$.  The pair $(W,S)$ is a \emph{Coxeter system}.   For a subset $T\le S$, the subgroup generated by $T$ is denoted $W_T$  and called a \emph{special subgroup}.  The pair $(W_T,T)$  also is a Coxeter system.  The subset $T$ is \emph{spherical} if $W_T$ is finite. Let $\cs(W,S)$ denote the poset of spherical subsets of $S$ (including the empty set).   There is a simplicial complex, denoted by $L(W,S)$  (or simply by $L$) called the \emph{nerve} of $(W,S)$.  Its simplices are the elements of $\cs(W,S)_{>\emptyset}$.  In other words, the vertex set of $L$ is $S$ and its simplices are precisely the nonempty spherical subsets of $S$.  Note that a subset $\{s,t\}< S$ of cardinality $2$ is an edge of $L$ if and only if $m_{st}\neq \infty$.  

Given a Coxeter matrix $(m_{st})$, there is
another  group $A$ called the associated \emph{Artin group}.  It has a generator $a_s$ for each vertex $s\in S$ and for each edge $\{s,t\}$ of $L$, an \emph{Artin relation}: 
	\begin{equation}\label{e:artin}
	\underbrace{a_sa_t\dots}_{m_{st} \text{ terms}}=\underbrace{a_ta_s\dots}_{m_{st} \text{ terms}}.
	\end{equation}
The \emph{special subgroup} $A_T$ is the subgroup generated by $\{a_s\mid s\in T\}$; it can be identified with the Artin group associated to $(W_T, T)$. The subgroup $A_T$ is \emph{spherical} if $T$ is a spherical.

A spherical Coxeter group $W_T$ acts as a group generated by linear reflections on $\rr^n$, where $n=\Card T$.  After tensoring with $\cc$, it becomes a linear reflection group on $\cc^n$.  Deligne \cite{deligne} proved that the complement of the arrangement of reflecting hyperplanes in $\cc^n$ is aspherical. Since $W_T$ acts freely on this complement, it follows that its quotient by $W_T$ is a model for $BA_T$.  Salvetti \cite{sal} described a specific $n$-dimensional CW complex, $\sal T$, called the \emph{Salvetti complex}, which is a model for this quotient of the hyperplane arrangement complement.  Thus, $\gdim A_T\le \dim (\sal T)=\Card T$. (In fact, since the cohomological dimension, $\cd A_T$, is $\ge \Card T$ the previous inequality is an equality.)  More generally, one can glue together the complexes $\sal T$, with $T$ spherical, to get a CW complex $\sal S$ with fundamental group $A$ ($=A_S$). The dimension of $\sal S$ is $\dim L + 1$.  The \emph{$K(\pi,1)$-Conjecture} for $A$ is the conjecture that $\sal S$ is a model for $BA$.  This conjecture is true in many cases, for example, whenever $L$ is a flag complex, cf.~\cite{cd1}. Thus, if the $K(\pi,1)$-Conjecture holds for $A$, then $\gdim A=\dim L+1$.

A Coxeter matrix $(m_{st})$  is \emph{right-angled} if  all its nondiagonal entries are either $2$ or $\infty$.  
The associated Artin group $A$ is a \emph{right-angled Artin group} (abbreviated $\raag$). 
The next theorem is one of our main results.  For $\raag$s, it  was proved in \cite{ados}.

\begin{mainthm}\textup{(Theorem~\ref{t:main} in section 6).}
Suppose  that $L$ is the nerve of a Coxeter system $(W,S)$ and that $A$ is the associated Artin group.  Let $d= \dim L$.  If $H_d(L;\zz/2)\neq 0)$, then $\actdim A\ge 2d+2$.  So, if the $K(\pi,1)$-Conjecture holds for $A$, then $\actdim A=2d+2$.
\end{mainthm}

The justification for the final sentence of this theorem is that if $K(\pi,1)$-Conjecture holds for $A$, then $\gdim A=d+1$; so, the largest possible value for the action dimension is $2d+2$.

For any simplicial complex $K$, there is another simplicial complex $OK$ of the same dimension, defined by ``doubling the vertices'' of $K$ (cf.\ subsection~\ref{ss:flats}). The complex $OK$ is called the ``octahedralization'' of $K$. In the right-angled case, the Main Theorem was proved by using $OL$ for an obstructor, where $L=L(W,S)$.

The principal innovation of this paper concerns a certain subdivsion $L_\oslash$ of $L$, called the ``complex of standard abelian subgroups'' of $A$.  The complex $L_\oslash$ plays the same role for a general Artin group as  $L$ does in the right-angled case; its simplices parameterize the ``standard'' free abelian subgroups in $A$.  We think of $OL_\oslash$ as the boundary of the union of the standard flat subspaces in the universal cover of $BA$.  The Main Theorem is proved by showing that
(1)  $\cone OL_\oslash$ coarsely embeds in $A$, and (2) when $H_d(L;\zz/2) \neq 0$, the van Kampen obstruction for $OL_\oslash$ in degree $2d$ is not zero.  

When $A$ is a $\raag$, a converse  to the Main Theorem  also was proved in \cite{ados}:  if $H_d(L;\zz/2)=0$ (and $d\neq 2$), then $\actdim A \le 2d+1$. For a general Artin group $A$ for which the $K(\pi,1)$-Conjecture holds, this converse, under a slightly weaker hypothesis, recently has  been proved in the PhD thesis of Giang Le \cite{le} :  if  $H^d(L;\zz)=0$ (and $d\neq 2$), then $\actdim A\le 2d+1$.

\section{Preliminaries}
\subsection{Coxeter systems}\label{ss:cox}
Let $(W,S)$ be a Coxeter system.

A \emph{reflection} in $W$ is the conjugate of an element of $S$.  Let $R$ denote the set of all reflections in $W$.  For any subset  $T<S$,  $R_T$ denotes the set of reflections in $W_T$, i.e., $R_T=R\cap W_T$.

The \emph{Coxeter diagram} $\bD$  of $(W,S)$ is a labeled graph which records the same information as does $(m_{st})$.   The vertex set of $\bD$ is $S$ and there is an edge between $s$ and $t$ whenever $m_{st}>2$.  If $m_{st}=3$, the edge is left unlabeled, otherwise it is labeled  $m_{st}$.  (The notation $\bD(W,S)$ or $\bD(S)$ also will be used for $\bD$.)

The Coxeter system is \emph{irreducible} if $\bD$ is connected.  A \emph{component} $T$ of $(W,S)$  is the vertex set of a connected component of $\bD$.  Thus, if $T_1, \dots, T_k$ are the components of $(W,S)$, then $\bD$ is the disjoint union of the induced graphs on $T_i$ and $W=W_{T_1}\times\cdots\times W_{T_k}$.  So, the  diagram shows us how to decompose $W$ as a direct product. In this paper we will only be concerned with the diagrams of spherical Coxeter groups.

\begin{definition}\label{d:irreducible}
Suppose $T$ is a spherical subset and that $T_1, \dots T_k$ are the vertex sets of the  components of $\bD(T)$.  Then $\{T_1,\dots T_k\}$ is the \emph{decomposition} of $T$ (into irreducibles).
\end{definition}

\begin{lemma}\label{l:longest}\textup{(cf. \cite{bourbaki}*{Ch.\,IV, Exerc.\,22, p.\,40}).}  
Suppose $(W_T,T)$ is a spherical Coxeter system.  Then there is a unique element $w_T$ of longest length in $W_T$.  It has the following properties:
\begin{itemize}
\item 
$w_T$ is an involution,
\item
$w_T$ conjugates $T$ to itself.
\item
$\ell(w_T) =\Card R_T$.
\end{itemize}
\end{lemma}

\begin{remark}\label{r:centerless}
Let $w_T$ be the element of longest length in a spherical special subgroup $W_T$. Let $\gi_T:W_T\to W_T$ denote inner automorphism by $w_T$.  By Lemma~\ref{l:longest}, $\gi_T$ restricts to a permutation of $T$ and this permutation induces an automorphism of the Coxeter diagram $\bD(T)$.  It follows that $w_T$ belongs to the center of $W_T$ if and only if the restriction of $\gi_T$ to $T$ is the trivial permutation.  Suppose $(W_T,T)$ is irreducible.  It turns out that the center of $W_T$ is trivial if and only if the permutation $\gi_T\vert_T$ is nontrivial (this is because this condition is equivalent to the condition that $w_T$ does not act as the antipodal map, $-1$,  on the canonical representation). If this is the case, we say $W_T$ is  \emph{centerless}.  In particular, when $W_T$ is centerless,  its Coxeter diagram admits a nontrivial involution.  Using this, the question of  the centerlessness of $W_T$  easily can be decided, as follows.
When $\Card T=2$,  $W_T$ centerless if and only if it is a dihedral group $\mathbf{I_2}(p)$, with $p$ odd.  When  $\Card T>2$, the groups of type $\mathbf{A}_n$, $\mathbf{D}_n$ with $n$ odd,  and $\mathbf{E}_6$ are centerless, while those of type $\mathbf{B}_n$, $\mathbf{D}_n$ with $n$ even, $\mathbf{H}_3$, $\mathbf{H}_4$, $\mathbf{F}_4$, $\mathbf{E}_7$ and $\mathbf{E}_8$ are not (cf. \cite[Appendices I-IX, pp.\,250-275]{bourbaki} or \cite[Remark 3.1.2, p.\,125]{djs2}). 
\end{remark}

\begin{remark} 
It is proved in \cite{qi} that if $W_T$ is irreducible and not spherical, then its center is trivial.
\end{remark}

For later use we record the following technical lemma.
\begin{lemma}\label{l:root}\textup{(The Highest Root Lemma).}
Suppose $(W,S)$ is  spherical.  Then there is a reflection $r\in R$ such that $r\notin W_T$ for any proper subset $T<S$.
\end{lemma}

\begin{proof}
It suffices to prove this when $(W,S)$ is irreducible.  So, suppose this.  If the diagram $(W,S)$ is not type $\mathbf{H}_3$, $\mathbf{H}_4$ or $\mathbf{I}_2(p)$ with $p=5$ or $p>6$, then $(W,S)$ is associated with a root system in $\rr^n$.  Each root $\phi$ can be expressed as an integral linear combination of the simple roots $\{\phi_s\}_{s\in S}$:
	\[
	\phi=\sum_{s\in S} n_s \phi_s,
	\]
where the coefficients $n_s$ are either all $\ge 0$ or all $\le 0$; the root $\phi$ is said to be  \emph{positive} or \emph{negative}, respectively.  Moreover, there is a positive root $\phi_r$ for each reflection $r\in R$.  It is proved in \cite[Ch.\,VI \S 1.8, Prop.\,25, p.\,178]{bourbaki} that there is always a highest root $\phi_r$ which dominates all other positive roots $\phi$, in the sense that the coefficients of $\phi_r-\phi$ are all $\ge 0$.  In particular, since $\phi_r$ dominates the simple roots,  the coefficients $n_s$ of $\phi_r$ are $\neq 0$.  On the other hand, if $r\in R_T$, then $n_s=0$ for all $s\in S-T$.  Hence, if $\phi_r$ is the highest root, then $r\notin R_T$ for any proper subset $T<S$.

When the diagram is type $\mathbf{H}_3$, $\mathbf{H}_4$ or $\mathbf{I}_2(p)$, one still has a ``root system''  with the properties in the previous paragraph,  except that the coefficients need not be integers.  One then can prove the lemma directly in each of these three cases by simply writing down a reflection $r$ which does not lie in any of the $R_T$. 
\end{proof}

Let $S_\oslash$ denote the set of irreducible nonempty spherical subsets of $S$.  In other words,
	\begin{equation}\label{e:soslash}
	S_\oslash :=\{T\in \cs(W,S)_{>\emptyset}\mid \bD(T)  \text{ is connected}\}.	
	\end{equation}
By Lemma~\ref{l:root}, one can choose a function $r:S_\oslash\to R$, denoted $T\mapsto r(T)$, such that 
	\begin{equation}\label{e:rT}
	r(T)\in R_T- \bigcup_{T'<T} R_{T'}
	\end{equation}
where the union is over all proper subsets of $T$.  (When $\bD (T)$ is type $\mathbf{A}_n$, it turns out that there is exactly one reflection in $R_T-\bigcup R_{T'}$; however,  in other cases, it can contain more than one reflection.)

\begin{lemma}\label{l:inj}
The function $r:S_\oslash \to R$ is injective.
\end{lemma}

\begin{proof}
According to \cite[Ch.\,IV, \S 1.8, Prop.\,7, p.\,12]{bourbaki} or \cite[Prop.\,4.4.1, p.\,44]{dbook}, for any given $w\in W$, there is a well-defined subset $S(w)\le S$, such that the set of letters used in any reduced expression for $w$ is precisely $S(w)$.  By construction, $S(r(T))=T$.  Hence, if $r(T_1)=r(T_2)$, then $T_1=T_2$, i.e., $r$ is injective.
\end{proof}

\subsection{The subdivision $L_{\oslash}$ of $L$.}  
We want to define the subdivision $L_\oslash$ of $L$ ($=L(W,S)$).    
First, the vertex set of $L_\oslash$ is the set $S_\oslash$ defined in \eqref{e:soslash}, i.e., $S_\oslash$ is the set of (vertex sets of) irreducible spherical subdiagrams of $\bD(W,S)$.  Note that $S_\oslash$ is a subposet of $\cs(W,S)$.  We  regard $T\in S_\oslash$ as being the barycenter of the corresponding simplex of $L$.    Given a subset  $\ga$ of $S_\oslash$,  define its \emph{support}: 
	\begin{equation*}\label{e:sp}
	sp (\ga) =\bigcup_{T\in \ga} T.
	\end{equation*}
Next we give an inductive definition of what it means for a subset of $S_\oslash$ to be ``nested''.

\begin{definition}
Let $\ga$ be a subset of $S_\oslash$ such that $sp(\ga)$ is spherical. By definition, the $\emptyset$ is \emph{nested}. 
(This is the base case of the inductive definition.) Let $\{T_1, \dots T_k\}$ be the set of maximal elements in $\ga$.  Then $\ga$ is \emph{nested} if 
\begin{itemize}
\item
$sp(\ga)$ is spherical.
\item
$\{T_1, \dots , T_k\}$ is the decomposition of $sp(\ga)$ into irreducibles.
\item
$\ga_{<T_i}$ is nested. (This is defined by induction since $\Card (\ga_{<T_i}) < \Card \ga$.)
\end{itemize}
\end{definition}

Elements $T_1$, $T_2$ of $S_\oslash$ are \emph{orthogonal} if their diagrams $\bD(T_1)$ and $\bD(T_2)$ are distance $\ge 2$ apart in $\bD(S)$; in other words, $\bD(T_1)$ and $\bD(T_2)$ lie in different components  of the full subgraph of $\bD(S)$ spanned by $\bD(T_1) \cup \bD(T_2)$.  Elements  $T_1$, $T_2$ are \emph{comparable} if either $T_1<T_2$ or $T_2<T_1$.  Elements $T_1$, $T_2$ are \emph{transverse} if they are neither orthogonal nor comparable.  It is easy to see that a two element subset $\{T_1, T_2\}$ of $S_\oslash$ is nested if and only if $T_1$ and $T_2$ are either orthogonal  or comparable.  

The simplicial complex $L_\oslash$ is defined as follows: the simplicies of $L_\oslash$ are the nested subsets of $S_\oslash$.  It is  clear that $L_\oslash$ is a subdivision of $L$.  (The subdivision of an irreducible $2$-simplex is shown in Figure \ref{fig:1}.) This same subdivision plays a role in \cites{djs1,djs2}.  

\begin{remark}
Here is a more geometric description of $L_\oslash$.  Suppose $T$ is a spherical subset and that $\gs$ is the corresponding geometric simplex in $L$.  Since $L_\oslash$ is defined by applying the subdivision procedure to each simplex of $L$, it suffices to describe $\gs_\oslash$ .  Suppose $\{T_1,\dots, T_k\}$ is the decomposition of $T$ into irreducibles and that $\gs_i$ is the geometric simplex in $L$ corresponding to $T_i$.  This gives a join decomposition: 
	\(
	\gs=\gs_1*\cdots *\gs_k\  
	\).  
Suppose by induction on $\dim \gs$ that the subdivision has been defined for each proper subcomplex of $\gs$.  The subdivision $\gs_\oslash$ is then defined by using one of the following two rules.
\begin{itemize}
\item
If $k=1$, then $T$ is irreducible and the barycenter $b_T$ of $\gs$ will be a vertex of the subdivision. Define $\gs_\oslash:=(\partial \gs)_\oslash * b_T$. (In other words, $\gs_\oslash$ is formed by coning off $(\partial \gs)_\oslash$ to $b_T$.)
\item
If $k>1$, then $\gs_\oslash := (\gs_1)_\oslash *\cdots * (\gs_k)_\oslash$.
\end{itemize}
\end{remark}

\begin{figure}[ht!]
	\centering
	\labellist
	\small\hair 2pt
	\pinlabel $\{a,b,c\}$ at 200 110
	\pinlabel $\{a\}$ at 185 274
	\pinlabel $\{b\}$ at 4 0
	\pinlabel $\{c\}$ at 314 0
	\pinlabel $\{b,c\}$ at 161 0
	\pinlabel $\{a,b\}$ at 61 152
	\endlabellist
	\includegraphics[scale=0.6]{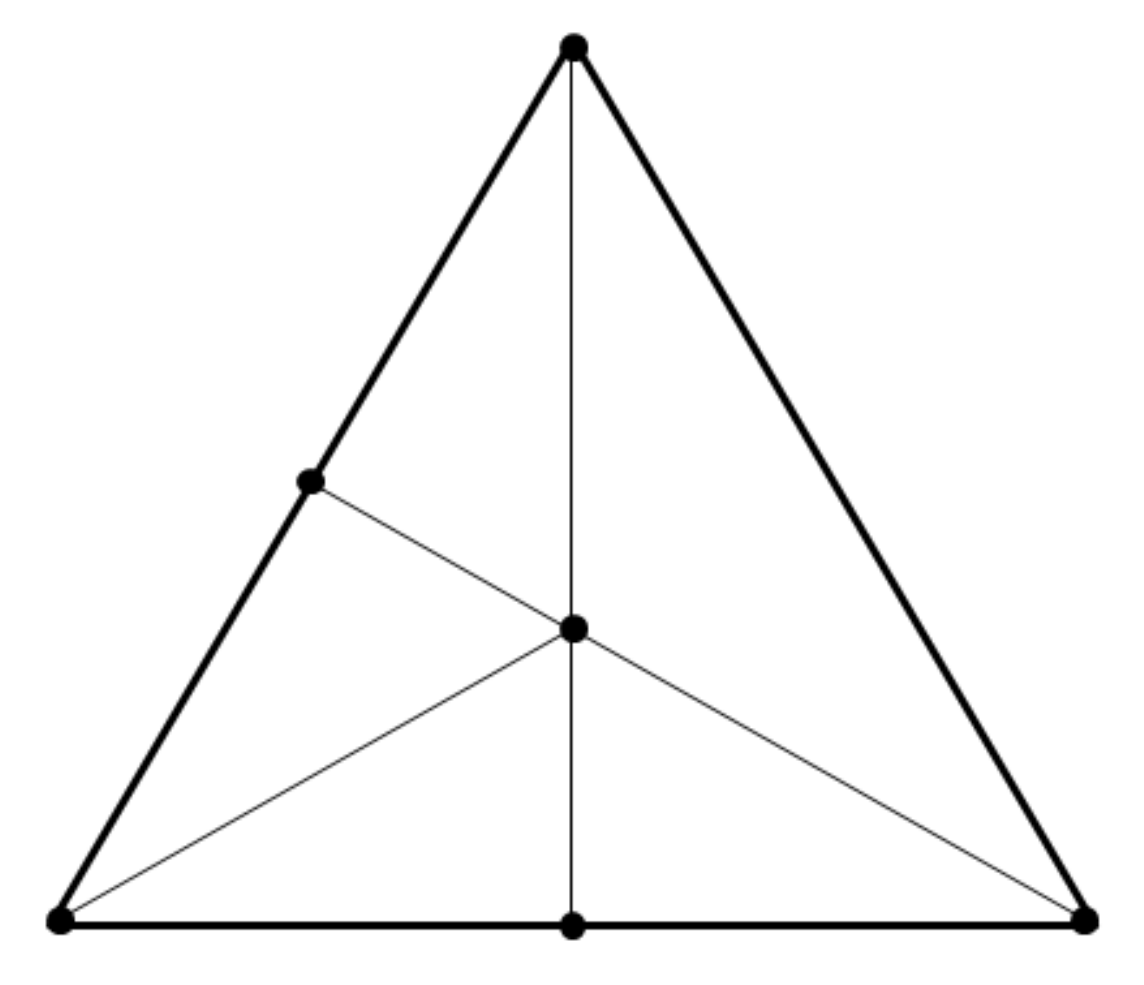}
	\caption{Subdivision for $\bD=$  \protect\includegraphics[scale=0.75]{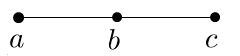}}
	\label{fig:1}
\end{figure}

The following lemma will be used in subsection \ref{ss:obstructor}.  Its proof is straightforward.
\begin{lemma}\label{l:flag}\textup{(cf. \cite[Cor.~3.5.4]{djs2}).}
The subdivision $L_\oslash$ is a flag complex.
\end{lemma}

\begin{remark}\label{r:edges}
Assuming this lemma, one can define $L_\oslash$ without using induction.  First, there is a direct description of the $1$-skeleton of $L_\oslash$:  the edges are the $2$-element subsets $\{T_1,T_2\}$ of $S_\oslash$ such that $T_1$ and $T_2$ are either orthogonal or comparable. Second, $L_\oslash$ is the flag complex determined by this $1$-skeleton.
\end{remark}

\subsection{Artin groups}\label{ss:artin}
Suppose $A$ is the Artin group associated to a Coxeter system $(W,S)$ and that $\{a_s\}_{s\in S}$ is its standard set of generators.  There is a canonical epimorphism $p:A\to W$ defined by $a_s\mapsto s$.  The \emph{pure Artin group} $PA$ is the kernel of $p$.  

There is a map $i:w\mapsto a_w$ from $W$  to  $A$, which is a set theoretic section for $p$.  It is defined as follows.  If $s_1\cdots s_l$ is a reduced expression for $w$, then 
	\begin{equation}\label{e:aw}
	a_w=a_{s_1}\cdots a_{s_k}
	\end{equation}
The element $a_w$ does not depend on the choice of reduced expression by  Tits' solution to word problem for Coxeter groups.  This asserts that any two reduced expressions for an element of $W$ differ by ``braid moves,''  corresponding to Artin relations of type \eqref{e:artin}, e.g., see \cite[\S 3.4, p.\,40]{dbook}.  (N.B. the map $i:W\to A$ is not a homomorphism.)

\begin{remark}
For subsets $T$, $T'$ subsets of $S$, we have
\[
W_T\cap W_{T'}= W_{T\,\cap\, T'}, \quad A_T\cap A_{T'}= A_{T\,\cap\, T'},\quad PA_T\cap PA_{T'}= PA_{T\,\cap\, T'}.
\]
The first formula is proved in \cite[Ch.\,IV.8, Thm.\,2, p.\,12]{bourbaki}, the second in 
\cite[]{lek}. The third follows from the first two.
\end{remark}

\subsection{Complements of hyperplane arrangements and \\pure Artin groups}\label{ss:pure}
When $T$ is a spherical subset of $S$, $W_T$ is a finite reflection group on $\rr^n$ (where $n$ is the cardinality of $T$) and so, by complexification, also a reflection group on $\cc^n$.  Let $M_T$ denote the complement of the union of reflecting hyperplanes in $\cc^n$.  Deligne \cite{deligne} proved that $\pi_1(M_T)$ is the pure Artin group $PA_T$ and that $M_T$ is a model for $BPA_T$.  Even when $W_T$ is infinite it still has a geometric representation as a reflection group on $\rr^n$ so that $W_T$ acts properly on the interior of a certain convex cone (the ``Tits cone'').  Hence, $W_T$ acts properly on some convex open subset $\gO_T<\cc^n$.  We also denote the complement of the union of reflecting hyperplanes in $\gO_T$ by $M_T$. It is proved in \cite{lek} that $\pi_1(M_T)\cong PA_T$.  The original version of the $K(\pi,1)$-Conjecture was that $M_T\sim BPA_T$.  (For more details, see \cite{cd1} or \cite{lek}.)

When $T'<T$, there is an open convex  subset in $\gO_T$ which only intersects the reflecting hyperplanes of $W_{T'}$ and which, therefore, can be identified with the product of $\gO_{T'}$ and a disk.  Taking complements of hyperplanes we get an inclusion $M_{T'}\hookrightarrow M_T$.  On the level of fundamental groups, this induces the standard inclusion $PA_{T'}\hookrightarrow PA_T$.  

Next, we consider the abelianization $PA^{ab}$ of $PA$ as well as the abelianization of the special subgroups $PA_T$. Since $PA$ is the fundamental group of the hyperplane arrangement complement $M_S$, we have  $PA^\ab= H_1(M_S)$.  There is one reflecting hyperplane for each reflection $r\in R$.  Since each such hyperplane is the intersection of a linear hyperplane with an open convex subset of $\cc^n$, $H_1(M_S)$ is generated by loops around these hyperplanes; in fact, $H_1(M_S)$ is the free abelian group on $R$, denoted $\zz^R$ (e.g., see \cite{ot}).  So, we have proved the following.

\begin{lemma}\label{l:abelianP}
$PA^\ab$ is the free abelian group $\zz^R$.  Similarly, for any $T<S$, $(PA_T)^\ab =\zz^{R_T}$.
\end{lemma}

Let $\{e_r\}_{r\in R}$ be the standard basis for $\zz^R$. For each fundamental reflection $s\in S$, the element $(a_s)^2\in PA$ is a loop around the reflecting hyperplane corresponding to $S$, i.e., $e_s$ is represented by $(a_s)^2$, where $a_s$ is a standard Artin generator.  Any reflection $r\in R$ can be written as $r=wsw^\minus$ for some $s\in S$ and $w\in W$.  Put
	\begin{equation}\label{e:gepsr}
	\geps_r:= (a_w) (a_s)^2 (a_w)^{-1},
	\end{equation}
where $a_w$ is the element of $A$ defined by \eqref{e:aw} (any other element of $A$ lying above $w$ would serve as well). Note that $\geps_r\in PA$ and that it projects to $e_r\in PA^\ab$. Thus, $\{e_r\}_{r \in R}$ is the standard basis for $\zz^R=PA^\ab$.  (Usually we will write the group operation in $PA^\ab$ additively.)

\section{Configurations of standard abelian subgroups}\label{s:config}

\subsection{The center of a spherical Artin group}\label{ss:artincenter}
Suppose $A_T$ is a spherical Artin group.  Define $\gD_T\in A$ by
	\begin{equation}\label{e:deltaT}
	\gD_T=a_{w_T},
	\end{equation}
where $w_T\in W_T$ is the element of longest length defined in Lemma~\ref{l:longest} and $a_{w_T}$ is defined by \eqref{e:aw}. For any $t\in T$, we have $w_Tt=\gi(t)w_T$, where $\gi:T\to T$ is the permutation induced by $w_T$.  Moreover, there is a reduced expression for $w_T$ ending with the letter $t$ as well as another reduced expression beginning with $\gi(t)$. So, $a_{w_T}(a_t)^\minus=(a_{\gi(t)})^\minus a_{w_T}$. Hence,
	\[
	\gD_T a_t=a_{\gi(t)}\gD_T.
	\]
It follows that $(\gD_T)^2$ lies in the center, $Z(A_T)$, of $A_T$ and that $\gD_T\in Z(A_T)$ if and only if $\gi:T\to T$ is the identity permutation.  If $T$ is spherical and irreducible,  put 
	\begin{equation*}\label{e:delta}
	\gd_T:=
	\begin{cases}
	\gD_T,	&\text{if $\gi=\id$;}\\
	(\gD_T)^2,	&\text{otherwise}.
	\end{cases}
	\end{equation*}
Note that $(\gD_T)^2\in PA_T$. The next lemma follows from \cite[Thm.\,4.7, p.\,294]{best99}.  
 
\begin{lemma}\label{l:center}
Suppose $(W_T,T)$ is spherical and irreducible.  Then $Z(A_T)$ is infinite cyclic with generator $\gd_T$.  Similarly, $Z(PA_T)$ is infinite cyclic with generator $(\gD_T)^2$.
\end{lemma}

If $(W_T,T)$ has more than one component, define $\gd_T$ to be the central element given by product of the appropriate choice of $\gD_{T'}$ or $(\gD_{T'})^2$ for each component $T'$.

We also can define corresponding  central elements in $PA_T$ and $(PA_T)^\ab$ by
	\begin{align*}
	\geps_T &:= (\gD_T)^2 \in PA_T, \text{ \ and} \\ 
	e_T&:= \text{\,image of  $(\gD_T)^2$ in } (PA_T)^\ab\ . 
	\end{align*}

\begin{lemma}\label{l:eT}
Suppose $T$ is spherical. Then $e_T\in \pat$ is the sum of the standard basis elements:
	\[
	e_T= \sum_{r\in R_T} e_r .
	\]
\end{lemma}
	
\begin{proof}
Consider the hyperplane arrangement complement, $M_T$.  If $\gD_T$ is defined by \eqref{e:deltaT}, then
the central element $(\gD_T)^2 \in \pi_1(M_T) = PA_T$ is represented by the ``Hopf fiber''.  (More precisely, the Hopf fiber is $\cc^*$ and $(\gD_T)^2$ is represented by $S^1< \cc^*$.)  The Hopf fiber links each reflecting hyperplane once. (To see this, consider the complement in $\cc^n$ of a single linear hyperplane, in which case it is obvious.)  Hence, the class of the Hopf fiber in $H_1(M_T)$  is $\sum_{r\in R_T} e_r$.
\end{proof}

Here is an alternative algebraic proof of the above lemma. 
\begin{proof}[Alternate proof of Lemma~\ref{l:eT}.]
Let $F(T)^{+}$ be the free monoid on $T$. 
For any word $f=t_1\cdots t_l$ which is a reduced expression for the corresponding element $w_f$ of $W_T$, let $a_f$ be the corresponding element of $A_T$, i.e., $a_f=a_{w_f}=a_1\cdots a_l$, where $a_i$ is the Artin generator corresponding to $t_i$. 
Pick a reduced expression $f\in F(T)^{+}$ for the element $w_T$ of longest length in $W_T$. Let $\overline{f}$ denote the reverse of $f$, i.e., if $f=t_1t_2\cdots t_k$, then $\overline{f}=t_k\cdots t_2t_1$. Note that 
	\begin{itemize}
	\item Since $f\cdot \overline{f}=1$ in $W_T$ and since $w_f=w_T$ is an involution (by Lemma~\ref{l:longest}),  $w_f=w_{\overline{f}}$. Since both $f$ and $\overline{f}$ are reduced expressions for $w_T$, the words $f$ and $\overline{f}$  represent the same element  in  $A_T$, namely, $\gD_T$.
	\item The subset of $W_T$ represented by $\{(t_1\cdots t_{i-1})\cdot t_i\cdot(t_1\cdots t_{i-1})^{-1} \}_{1\le i\le k}$ is exactly the set of reflections in $W_T$ (cf.\ \cite[Lemma 2, p.\,6]{bourbaki}).
	\end{itemize}
A simple calculation shows  that $\geps_T = (\gD_T)^2 = a_f\cdot a_{\overline{f}} = x_k x_{k-1}\cdots x_1$ where $x_i=(a_1\cdots a_{i-1})\cdot (a_i)^2 \cdot (a_1\cdots a_{i-1})^{-1}$ for $1\le i\le k$ (cf.\ \eqref{e:gepsr}). By the discussion in subsection \ref{ss:pure}, the images of the $x_i$'s in $(PA_T)^{ab}$ are exactly the $e_r$'s with $r\in R_T$.  The lemma follows.
\end{proof}

Let $j:\zz^{S_\oslash}\to \zz^R=PA^\ab$ be the homomorphism defined  by $T\mapsto e_T$  and extending linearly.
By combining Lemmas~\ref{l:inj} and \ref{l:eT}, we get the following.

\begin{lemma}\label{l:homo}
The homomorphism $j:\zz^{S_\oslash} \to \zz^R$ is injective.
\end{lemma}

\begin{proof}
The standard inner product, $(x,y)\mapsto x\cdot y$ on $\zz^R$ is defined on standard basis elements by $(e_r,e_{r'})\mapsto \gd_{r,r'}$.  For each $T\in S_\oslash$ let $r(T)$ be as in \eqref{e:rT} and Lemma~\ref{l:inj}.  Note that for any $T\in S_\oslash$ and $r\in R$, we have $e_T\cdot e_r=1$ if $r\in R_T$ and is $0$ otherwise. It follows that for $T, T'\in S_\oslash$, we have:
	\[
	e_{T'}\cdot e_{r(T)}=
	\begin{cases}
	1,	&\text{ if $T\le T'$,}\\
	0,	&\text{ otherwise.}
	\end{cases}
	\] 
Using this one sees that the $e_T$ are linearly independent in $\zz^R$.  Indeed, if $\sum x_T e_T =0$, after taking the inner product with $e_{r(T)}$, we get $x_T=0$ whenever $T$ is a maximal element of $S_\oslash$. Let $S'_\oslash$ be $S_\oslash$ with its maximal elements removed. Repeat the argument for maximal elements of $S'_\oslash$. After finitely many iterations, we get that $x_T=0$ for all $T\in S_\oslash$.
\end{proof}

\subsection{Standard abelian subgroups}\label{ss:abelian}
Suppose $\ga$ is a $d$-simplex in  $L_\oslash$.  We are going to define a subgroup $H_\ga  <A$ which is free abelian of rank $d+1$.  Suppose $T,T'\in \vertex \ga$ (where $\vertex \ga$ means the vertex set of $\ga$). By Remark~\ref{r:edges},  $T$ and $T'$ are either orthogonal or comparable. In either case $\gd_T$ and $\gd_{T'}$ commute. (This is obvious if  $T$ and $T'$ are orthogonal.  If $T$ and $T'$ are comparable, then without loss of generality $T<T'$, in which case,  $\gd_{T'}$ is centralizes $A_T<A_{T'}$.)  

Define  $H_\ga$ to be the subgroup of $A_{sp (\ga)}$ generated by $\{\gd_T\}_{T\in \vertex \ga}$.  By the previous paragraph, $H_\ga$ is abelian.  It is called the \emph{standard abelian subgroup} associated to $\ga$.  Similarly, define $PH_\ga$ to be the subgroup of finite index in $H_\ga$ generated by $\{(\gD_T)^2\}_{T\in \vertex \ga}$.  Also, denote by $J_\ga$ the image of $PH_\ga$ in $(PA_{sp(\ga)})^\ab$.

\begin{lemma}\label{l:iso}
The natural map $\zz^{\vertex \ga} \to H_\ga$ is an isomorphism; so, $H_\ga$ is free abelian of rank $(d+1)$.
\end{lemma}

\begin{proof}
It suffices to show that the homomorphism  $\zz^{\vertex \ga}\to (PH_\ga)^\ab$ induced by $T\mapsto e_T$ is an isomorphism.  This is immediate from Lemma~\ref{l:homo}.
\end{proof}

\begin{corollary}\label{c:cd}
If $\dim L = d$, then $\cd A\ge d+1$. If the $K(\pi,1)$-Conjecture holds for $A$, then $\gdim A=d+1$. 
\end{corollary}

\begin{proof}
If $\ga$ is a $d$-simplex in $L_\oslash$, then $H_\ga\cong \zz^{d+1}$, which has cohomological dimension $d+1$.  Since $\cd A \ge \cd H_\ga$, this proves the first sentence.  If the $K(\pi,1)$-Conjecture holds for $A$, then its Salvetti complex is a model for $BA$ and consequently, $d+1\ge \gdim A\ge \cd A$.
\end{proof}

Let $j:\zz^{S_\oslash}\to \zz^R=PA^\ab$ be the homomorphism defined in the sentence preceding Lemma~\ref{l:homo}. Let   $J:=\Ima j <PA^\ab$ denote its image.  By Lemma~\ref{l:homo}, $j:\zz^{S_\oslash}\to J$ is an isomorphism.  So, for any simplex $\ga$ of $L_\oslash$,  $J_\ga =j(\zz^{\vertex \ga})$. The next result is a key lemma.

\begin{lemma}\label{l:main}
If $\ga$ and $\gb$ are simplices of $L_\oslash$, then $H_\ga \cap H_\gb = H_{\ga\,\cap\, \gb}$.
\end{lemma}
\begin{proof}
Since $PH_\ga$ has finite index in $H_\ga$,
the lemma will follow once we show that $PH_\ga\cap PH_\gb = PH_{\ga\,\cap\, \gb}$. 
Let $h:PA\to PA^{ab}$ be the abelianization homomorphism. Recall that $h(PH_{\alpha})=J_{\alpha}$ and $h|_{PH_{\alpha}}$ is injective (Lemma \ref{l:iso}). It follows that if $PH_\ga\cap PH_\gb \supsetneq PH_{\ga\,\cap\, \gb}$, then $J_{\alpha}\cap J_{\gb}\supsetneq J_{\ga\,\cap\, \gb}$. So it suffices to show $J_{\alpha}\cap J_{\gb}=J_{\ga\,\cap\, \gb}$. But this is immediate from the equality:
	\[
	\zz^{\vertex \ga} \cap \zz^{\vertex \gb} = \zz^{\vertex (\ga\,\cap\,\gb)},
	\] 
after using the isomorphism $j$ to transport this equality to an equality involving subgroups of $J$.
\end{proof}

We record the following lemma concerning the geometry of standard abelian subgroups, eventhough we do not need  it in the sequel.
\begin{lemma}
	\label{l:qi}
For any simplex $\alpha\subset L_\oslash$, $H_{\alpha}$ is quasi-isometrically embedded in $A$.
\end{lemma}
	
\begin{proof}
First, consider the case when $A$ is spherical. Charney \cite{charney} proved that spherical Artin groups are biautomatic and Gersten-Short \cite{gs91} proved that biautomatic groups are semihyperbolic.  The Algebraic Flat Torus Theorem from \cite[p.\ 475]{bh} states that  every monomorphism of a finitely generated abelian group to a semihyperbolic group is a quasi-isometric embedding. The general case follows from the spherical case, since any standard abelian subgroup is contained in a spherical special subgroup and special subgroups are convex in $A$ \cite{cp}. 
\end{proof}

\subsection{Configurations of flats}\label{ss:flats}
\paragraph{Octahedralization} 
For a finite set $V$,  let $O(V)$ denote the boundary complex of the octahedron (or cross polytope) on $V$.
In other words, $O(V)$ is the simplicial complex with vertex set $V\times \{\pm 1\}$ such that a subset $\{(v_0, \geps_0),\dots, (v_k, \geps_k)\}$ of $V\times \{\pm1\}$ spans a $k$-simplex if and only if its first coordinates $v_0,\dots v_k$ are distinct.
Projection to the first factor $V\times \{\pm1\}\to V$ gives a simplicial projection $p:O(V)\to \gs(V)$, where $\gs(V)$ means the full simplex on $V$.
Any finite simplicial complex $K$ with vertex set $V$ is a subcomplex of $\gs(V)$. Its
 \emph{octahedralization} $OK$ is defined to be the inverse image of $K$ in $O(V)$:
	\begin{equation*}\label{e:ok}
	OK:= p^\minus (K) \le O(V).
	\end{equation*}
Thus, each simplex of $OK$ is of the form $(\gs,\geps)$, where $\gs$ is a simplex of $K$ and $\geps:\vertex{\gs}\to \{\pm 1\}$ is a function.
(This terminology comes from \cite{ados} or \cite{do12}*{\S 8}.) We also say that $OK$ is the result of  \emph{doubling the vertices of $K$}.  

Given a finite simplicial complex or metric space $K$, put
	\begin{equation}\label{e:cone}
	\cone(K):=K\times [0,\infty)\,/\, K\times \{0\}.
	\end{equation}
If $K$ is a metric space, then there is an induced metric on $\cone(K)$ which gives it the structure of a \emph{euclidean cone} (cf.\ \cite{bh}*{p.\,60}).  The idea is to use the formula  for the euclidean metric on $\rr^n$ in polar coordinates - the coordinate in $K$ is the ``angle'' (and the metric on $K$ gives the angular distance) and $[0,\infty)$ gives the radial coordinate.  For example, if $K$ is the round $(n-1)$-sphere, then $\cone(K)$ is isometric to $\ee^n$; if $K$ is a spherical $(n-1)$-simplex, then $\cone(K)$ is isometric to the corresponding  sector in $\ee^n$.

A spherical $(n-1)$-simplex $\gs$ is \emph{all right} if its edges all have length $\pi/2$. This means that $\gs$ is isometric to the intersection of the unit sphere in $\rr^n$ with the positive orthant $[0,\infty)^n$. Its octahedralization $O\gs$ is a triangulation of $S^{n-1}$ and with the induced all right piecewise spherical metric, it is isometric to the round sphere.  Hence,  $\cone (O\gs)$ is isometric to euclidean space.

Suppose $K$ is given its all right piecewise spherical structure in which each edge has length $\pi/2$.  Then $\cone(K)$ is isometric to a configuration of the positive orthants $\cone (\gs)$ and $\cone (OK)$ is isometric to a configuration of euclidean spaces all passing through the cone point. (There is one euclidean space for each $\gs \in K$.)  The link of the cone point in $\cone (OK)$ is $OK$.  By Gromov's Lemma (cf.\ \cite{dbook}*{Appendix I.6, p.\,516}), $\cone(OK)$ is $\cat (0)$  if and only if $OK$ is a flag complex.

\paragraph{Coordinate subspace arrangements}  Given a finite simplicial complex $K$, one can define an arrangement of coordinate subspaces in the euclidean space $\rr^{\vertex K}$.   Let $\cs (K)$ denote the poset of simplices in $K$.  For each $\gs\in \cs(K)$, there is the subspace $\rr^{\vertex \gs} \le \rr^{\vertex K}$.  The \emph{coordinate subspace arrangement} corresponding  to $K$ is the collection of subspaces, $\ca(K):=\{\rr^{\vertex \gs}\}_{\gs\in \cs (K)}$.  The intersection of these subspaces with the integer lattice $\zz^{\vertex K}$ gives the \emph{coordinate arrangement of free abelian subgroups} associated to $K$, 
	\begin{equation}\label{e:coord}
	\ca_\zz(K):=\{\zz^{\vertex \gs}\}_{\gs\in \cs(K)}. 
	\end{equation}
The union of the elements of $\ca(K)$ or $\ca_\zz(K)$ is denoted $\ee^{OK}$ or $\zz^{OK}$, respectively:
	\begin{equation}\label{e:zok}
	\ee^{OK}:= \bigcup_{\gs\in \cs(K)} \rr^{\vertex \gs}, \qquad\quad 
	\zz^{OK}:= \bigcup_{\gs\in \cs(K)} \zz^{\vertex \gs}.
	\end{equation}
Let $\{b_v\}_{v\in \vertex K}$ be the standard basis for $\rr^{\vertex K}$ so that $\{b_v\}_{v\in \vertex \gs}$ is the standard basis for $\rr^{\vertex \gs}$.  The positive cone spanned by standard basis $\{b_v\}_{v\in \vertex \gs}$ is the \emph{positive orthant}, $[0,\infty)^{\vertex \gs}<\rr^{\vertex \gs}$.  Similarly, if $(\gs,\geps)$ is a simplex of $OK$, where $\geps$ is a choice of signs $v\mapsto \geps_v\in \{\pm1\}$, we get the orthant spanned by $\{\geps_v b_v\}_{v\in \vertex \gs}$.

The proof of the next lemma is immediate from the definitions.

\begin{lemma}\label{l:arrang}
Let $\{h_v\}_{v\in \vertex K}$ be a collection of elements in some group $\pi$ indexed by vertices of a finite simplicial complex $K$. Suppose $h_{v}$ and $h_{v'}$ commute whenever $v$ and $v'$ are joined by an edge. For a simplex $\sigma$ in $K$, let $H_{\sigma}\le \pi$ denote the image of the homomorphism $\zz^{\vertex \sigma}\to \pi$ induced by $v\mapsto h_v$. Suppose that
	\begin{itemize}
	\item the homomorphism $\zz^{\vertex \sigma}\to \pi$ is injective,
	\item the conclusion of Lemma~\ref{l:main} holds, i.e., $H_\gs\cap H_\gt=H_{\gs\,\cap\,\gt}$ for all $\gs$, $\gt$ in $K$.
	\end{itemize}
Then the collection of free abelian subgroups $\ca=\{H_\gs\}_{\gs\in \cs(K)}$ is isomorphic to the coordinate arrangement of abelian groups in $\zz^{OK}$ in the obvious sense.
\end{lemma}

By Lemmas~\ref{l:iso} and \ref{l:main}, the  lemma above applies to the  simplicial complexes $K= L_\oslash$ and $OK=OL_\oslash$.  For each $\ga\in L_\oslash$, let
	\begin{equation*}\label{e:basis}
	\cb_\ga= \{\gd_T\}_{T\in \vertex \ga}
	\end{equation*}
 be the standard basis for $H_\ga$ and let $\rr^{\vertex \ga}$ be the euclidean space with basis $\cb_\ga$, i.e., $\rr^{\vertex \ga}=H_\ga\otimes \rr$. 
 
\begin{remark}
These notions also make sense when $K$ is infinite and $\mathcal{A}$ is an infinite arrangement of free abelian subgroups in $\pi$. If $\mathcal{A}$ is invariant under conjugation, then the conjugation action induces an action $\pi$ on $OK$. In several cases, such action plays an important role of understanding the group $\pi$. For example, if $\pi$ is the mapping class group of closed hyperbolic surface and $\mathcal{A}$ is the collection of abelian subgroups generated by Dehn twists, then $K$ is the curve complex of the underlying surface. If $\pi$ is a right-angled Artin group and $\mathcal{A}$ is the collection standard abelian subgroups and their conjugations, then $K$ is the flag complex of the ``extension graph,'' which was introduced by Kim and Koberda \cite{kimko}.
\end{remark}

\section{The obstructor dimension}\label{s:obdim}
In this section we review some definitions from \cite{bkk}.
\subsection{The meaning of ``$K\subset \partial \pi$''}\label{ss:meaning}
For a finite simplicial complex $K$, define $\cone(K)$ as in \eqref{e:cone}. We give $\cone(K)$ the euclidean cone metric after choosing a piecewise euclidean metric on $K$.  Next we triangulate $\cone(K)$ so that 
\begin{itemize}
\item
for each simplex $\gs$ of $K$, $\cone (\gs)$ is a subcomplex,
\item
the edge-length metric on the $0$-skeleton, $\cone(K)^0$, is bi-Lipschitz equivalent to the metric on $\cone(K)^0$ induced from the euclidean cone metric.
\end{itemize}  

\begin{definition}\label{d:boundary}
Suppose $\pi$ is a discrete, finitely generated group equipped with a word metric.  As in \cite[p.\ 228]{bkk}, write ``$K\subset \partial \pi$'' to mean that there is a proper, expanding, Lipschitz map, $\cone (K)^0\to \pi$.
\end{definition}

\noindent
The terms ``proper'' and ``Lipschitz'' have their usual meanings.  The term ``expanding'' needs further explanation, which is given below.

Two functions $h_1:A_1\to B$ and $h_2:A_2\to B$ to a metric space $B$ \emph{diverge} from one another if for every $D>0$ there are compact subsets $C_i\le A_i$, $i=1,2$, such that $h_1(A_1-C_1)$ and $h_2(A_2-C_2)$ are of distance $>D$ apart. 
Following \cite{bkk}, call a proper map $h:\cone(K)\to B$ is \emph{expanding} if for any two disjoint simplices $\gs$ and $\gt$ in $K$, the 
maps $h\vert_{\cone(\gs)}$ and $h\vert_{\cone(\gt)}$ diverge.  Similarly,  if $\cone(K)^0$ is equipped with an edge-length metric  as above, then $h:\cone(K)^0\to B$ is \emph{expanding} if for disjoint simplices $\gs$, $\gt$ in $K$, $h\vert_{\cone(\gs)^0}$ and $h\vert_{\cone(\gt)^0}$ diverge.

Note that the meaning of ``$K\subset \partial \pi$'' in Definition \ref{d:boundary} does not depend on choices we made in the beginning of this subsection.

It is shown in \cite{bkk} that if the universal cover of $B\pi$ has a $\mathcal{Z}$-structure (i.e., a $\mathcal{Z}$-set compactification) with boundary denoted  $\partial \pi$ and if $K\subset \partial \pi$, then there is a proper, expanding, Lipschitz map $\cone(K)\to \pi$.  In other words, $K\subset \partial \pi$  $\implies$ ``$K\subset \partial \pi$''.

\subsection{Obstructors}\label{ss:obstructor}
Let $K$ be a finite simplicial complex.  The \emph{deleted simplicial product}, $(K\times K)-\gD$, is the union of all cells in $K\times K$ of the form $\gs\times \gt$, where $\gs$ and $\gt$ are disjoint closed simplices in $K$.  The \emph{configuration space}, $\cac(K)$ is the quotient of $(K\times K)-\gD$ by the free involution which switches the factors.  Let $c:\cac(K)\to \rr P^\infty$ classify the double cover. If $w\in H^1(\rr P^\infty;\zz/2)$ is the generator, then the \emph{van Kampen obstruction} for $K$ in degree $m$ is the cohomology class $\vktwo^m(K)\in H^m(\cac(K);\zz/2)$ defined by 
	\begin{equation*}\label{e:vk}
	\vktwo^m(K) := c^*(w^m).
	\end{equation*}
It  is an obstruction to embedding $K$ in $S^m$.  Moreover, in the case $m=2\dim K$ with $\dim K\neq 2$, it (or actually an integral version of it) is the complete obstruction to embedding $K$ in $S^m$. If $\vktwo^m(K)\neq 0$, then $K$ is an \emph{$m$-obstructor}.

The main results of \cite{ados} concern the van Kampen obstruction $\vktwo^m(OK)$ of an octahedralization.  From a nonzero $\zz/2$-cycle  $M\in Z_k(K;\zz/2)$ and a $k$-simplex $\gD$ in the support of $M$, one produces a $2k$-chain $\gO\in C_{2k}(\cac(OK);\zz/2)$ on which one can evaluate the van Kampen obstruction.  If the pair $(M,\gD)$ satisfies a certain technical condition, called the ``$*$-condition''  in \cite{ados}*{p.\,122}, then $\gO$ is a cycle and $\vktwo^{2k}(OK)$ evaluates nontrivially on it.  This gives the following.

\begin{Theorem}\label{t:ados}\textup{(\cite{ados}*{Thm.~5.2, p.\ 121}).}
Given a simplicial complex $K$, suppose there is a $k$-cycle $M\in Z(K;\zz/2)$ and a $k$-simplex $\gD$ in $M$ such that $(M,\gD)$ satisfies the $*$-condition.  Then $\vktwo^{2k}(OK)\neq 0$, i.e., $OK$ is a $2k$-obstructor. 
\end{Theorem}

It is then remarked in \cite{ados} that when $k=\dim K$ and $K$ is a flag complex, the $*$-condition is automatically satisfied.  It also is proved in \cite[Thm.~5.1]{ados} that when  there is no nontrivial cycle on $K$ in the top degree, $\vktwo^{2k}(OK)=0$.  So, the following is a corollary of Theorem~\ref{t:ados}.

\begin{Theorem}\label{t:ados1}\textup{(\cite{ados}*{Thm.~5.4, p.\ 121}).}
Suppose $K$ is a $d$-dimensional flag complex.  Then $\vktwo^{2d}(OK)\neq 0$ if and only if $H_d(K;\zz/2)\neq 0$.  In other words, $OK$ is a $2d$-obstructor if and only if $H_d(K;\zz/2)\neq 0$.
\end{Theorem}

An immediate corollary is the following.

\begin{Theorem}\label{t:ados2}
Suppose a $d$-dimensional complex $L$ is the nerve of $(W,S)$. Then $OL_\oslash$  is a $2d$-obstructor if and only if $H_d(L;\zz/2)\neq 0$.
\end{Theorem}

\begin{proof}
Theorem~\ref{t:ados1} applies, since by Lemma~\ref{l:flag}, $L_\oslash$ is a flag complex. 
\end{proof}

\begin{remark}\label{r:schreve}
Kevin Schreve has pointed out to us that below the top dimension, it is easier for the $*$-condition to be satisfied for $L_\oslash$ than for $L$.  Indeed,  suppose $M$ is a $k$-cycle on $L$ and $\gD$ is a $k$-simplex in $M$.  The subdivision $M_\oslash$ is homologous to $M$.  Moreover, if the barycenter of $\gD$ occurs in $\gD_\oslash$, then the $*$-condition  holds for $(M_\oslash, \gD')$ where $\gD'$ is any $k$-simplex in $\gD$ containing the barycenter.
\end{remark}

\subsection{The definition of obstructor dimension}\label{ss:obdim}
\begin{definition}\label{d:obdim}\textup{(\cite{bkk}*{p.~225}).}
The \emph{obstructor dimension of $\pi$}, denoted by $\obdim \pi$, is $\ge m+2$ if there is an $m$-obstructor $K$ such that ``$K\subset \partial\pi$'' , i.e., 
	\[
	\obdim \pi:=\sup \{m+2\mid \exists \text{ an $m$-obstructor $K$ with ``$K\subset \partial\pi$''}\}
	\]
\end{definition}

If there is a proper, expanding, Lipschitz map from $\cone(K)$ into a contractible $(m+1)$-manifold, then $\vktwo^m(K)=0$. 
So,  if ``$K\subset \partial\pi$'' and $\pi$ acts properly on a contractible $(m+1)$-manifold, then  the van Kampen obstruction of $K$ vanishes in degree $m$ (cf.\ \cite{bkk}*{Thm.\,15, p.\,226}); hence,
	\begin{equation}\label{e:actinequality}
	\obdim \pi \le \actdim \pi .
	\end{equation}
	
\begin{example}\label{ex:raags}($\raag$s, cf.~\cite{ados}).  Suppose $A$ is a $\raag$ with $d$-dimensional nerve $L$.  The standard model for $BA$ is a union of tori; moreover, it is a locally $\cat(0)$ cubical complex of dimension $d+1$. Its universal cover $EA$ is $\cat(0)$, hence, contractible.  Therefore, $\gdim A =d+1$. Consequently, $\actdim A\le 2(d+1)$.  The lifts of the tori in $BA$ to $EA$, which contain  a given base point, give an isometrically embedded copy of $\ee^{OL}$ in $EA$. Thus, $\OL\subset \partial A$.  First suppose that $H_d(L;\zz/2)\neq 0$.  Since $L$ is a flag complex, Theorem~\ref{t:ados} implies that $OL$ is a $2d$-obstructor; so, $\obdim A\ge 2d+2$.  This gives $\actdim A \le 2d+2\le \obdim A$.  By  \eqref{e:actinequality}, both inequalities are equalities:
	\[
	\actdim A=2d+2=\obdim A .
	\]
Conversely, if $H_d(L;\zz/2)=0$, then $OL$ is not a $2d$-obstructor.  So,  when $d\neq 2$, $OL$ embeds in $S^{2d}$.  An argument of \cite{ados}*{Prop.\,2.2} then shows that the right-angled Coxeter group  with nerve $OL$, denoted $W_{OL}$,  acts properly on a contractible $(2d+1)$-manifold.  Since $A< W_{OL}$, so does $A$.  So, in this case, $\obdim A\le \actdim A\le 2d+1$. 
\end{example}

\section{Proper, expanding, Lipschitz maps}\label{s:proper}

\subsection{Coarse intersections of subgroups}\label{ss:coarse}
We will need  a lemma from \cite{msw}.   Before stating it as  Lemma~\ref{l:coarse} below, we explain some terminology.

Subspaces $A$ and $B$ of  a metric space $X$ are \emph{coarsely equivalent}, denoted $A\stackrel{c}{=}B$, if their Hausdorff distance is finite. The subspace $A$ is coarsely contained in $B$, denoted  $A\stackrel{c}{\subset}B$, if there is a positive real number $R$ such that $A$ is contained in the $R$-neighborhood $N_{R}(B)$ of $B$. In this case, we also write $A\subset_R B$. We use $A\cap_{R}B$ to mean $N_R(A)\cap N_R(B)$. A subspace $C$ is the \emph{coarse intersection} of $A$ and $B$, denoted by $C=A\stackrel{c}{\cap}B$, if the Hausdorff distance between $C$ and $A\cap_{R}B$ is finite for all sufficiently large $R$. In general, the coarse intersection of $A$ and $B$ may not exist; however, if it exists, then it is unique up to coarse equivalence.

The word metric (with respect to some finite generating set) on a finitely generated group $\pi$ gives it the structure of a metric space. Different choices of a generating set give rise to bi-Lipschitz equivalent word metrics. So, the various notions of  ``coarseness'' explained in the previous paragraph make sense for subsets of $\pi$. 

\begin{lemma}\label{l:coarse} \textup{(\cite[Lemma 2.2]{msw}).}
\label{coarse intersection subgroups}
Let $H_1,H_2$ be subgroups of a finitely generated group $\pi$. The coarse intersection of $H_1$ and $H_2$ exists and is coarsely equivalent to $H_1\cap H_2$.
\end{lemma}

\subsection{Subcomplexes at infinity}\label{ss:a}
Suppose $H$ is a finitely generated subgroup of a finitely generated group $\pi$.  With respect to the word metrics, the  inclusion $H\hookrightarrow \pi$ obviously is Lipschitz.  The  next lemma, a standard result, implies that $H\hookrightarrow \pi$  is a proper map.  For a proof, see \cite{des}.

\begin{lemma} \label{l:proper} \textup{(\cite[Lemma 3.6]{des}).} 
Suppose $H$ is a finitely generated subgroup of a finitely generated group $\pi$. Then there exists a function $\rho:[0,\infty)\to [0,\infty)$. with $\lim_{t\to\infty}\rho(t)=+\infty$ such that for any $a,b\in H$, if $d_{H}(a,b)\ge t$, then $d_{\pi}(i(a),i(b))\ge\rho(t)$.  (Here $d_{\pi}$ and $d_{H}$ denote the word metrics on the corresponding groups.)
\end{lemma}

We return to the situation of subsection~\ref{ss:flats}:  
$\ca=\{H_\gs\}_{\gs\in \cs(K)}$ is an arrangement of free abelian subgroups in a finitely generated  group $\pi$, indexed by the poset of simplices of a finite simplicial complex $K$.   Put $\ch =\bigcup H_\gs$.   
Further suppose $\ca$ is isomorphic to $\ca_\zz(K)$, the coordinate arrangement of free abelian groups in $\zz^{OK}$, defined by \eqref{e:coord} and \eqref{e:zok}.    The word metrics on the $H_\gs$ induce a metric on $\ch$. Since $H_\gs$ is bi-Lipschitz  to $\zz^{\vertex \gs}$, the  isomorphism $\cone (OK)^0\to \ch$ also is bi-Lipschitz.  

\begin{Theorem}\label{t:OK}
Suppose $\ca=\{H_\gs\}_{\gs\in \cs(K)}$ is an arrangement of free abelian subgroups isomorphic to the coordinate arrangement in $\zz^{OK}$. Then ``\,$OK\subset \partial \pi$''.
\end{Theorem}

By Lemma~\ref{l:arrang}, the arrangement of standard abelian subgroups in $A$, $\{H_\ga\}_{\ga\in \cs(L_\oslash)}$, is isomorphic to the coordinate arrangement $\ca_\zz(L_\oslash)$ in $\zz^{OL_\oslash}$.  So, the following theorem is a corollary of Theorem~\ref{t:OK}.  It is one of our main results.  

\begin{Theorem}\label{t:OL}
``\,$\OL_\oslash \subset \partial A$''.
\end{Theorem}

\begin{proof}[Proof of Theorem~\ref{t:OK}]
Let $i:\cone(OK)^0\cong \ch \hookrightarrow \pi$ be the inclusion. As we explained above,  $i$ is Lipschitz and proper. The  remaining issue is to show that it is expanding. So, suppose $\Delta=(\gs,\geps)$ and $\Delta'=(\gs',\geps')$ are disjoint simplices of $OK$. Let $\sphere$ and $\sphere'$ denote  the octahedral spheres $O\gs$ and $O\gs$ inside $OK$. It suffices to show
	\begin{equation}
	\label{e:diverge}
	i(\cone(\Delta)^{(0)})\cap_r i(\cone(\Delta')^{(0)}) \textmd{\ is\ finite\ for\ any\ }r>0.
	\end{equation}
		
	\noindent
	\textbf{Case 1:}  $\sphere=\sphere'$. In this case $\sigma=\sigma'$ (and since $(\gs,\geps)$ and $(\gs', \geps')$ are disjoint, $\geps'=-\geps$). Condition (\ref{e:diverge}) follows from Lemma \ref{l:proper} applied to the monomorphism $i:\zz^{\vertex\sigma}\to \pi$.  Given any two subcomplexes $F_1$ and $F_2$ of $\sphere=O\sigma$, We also deduce from Lemma \ref{l:proper} that  for any $r>0$, there exists $r'>0$ such that 
	\begin{equation}
	\label{e:intersection}
	i(\cone(F_1)^{(0)})\cap_r i(\cone(F_2)^{(0)})\subset_{r'} i(\cone(F_1\cap F_2)^{(0)}).
	\end{equation}
	When $F_1\cap F_2=\emptyset$, $\cone(F_1\cap F_2)^{(0)}$ is just the cone point.
	\smallskip
	
	\noindent
	\textbf{Case 2:} $\sphere\neq \sphere'$. By Lemma \ref{coarse intersection subgroups}, 
	\begin{align}
	\label{e:intersection1}
	i(\cone(\Delta)^{(0)})\cap_r i(\cone(\Delta')^{(0)}) &\subset i(\cone(\sphere)^{(0)}) \cap_r i(\cone(\sphere')^{(0)}) \nonumber\\
	&\subset_{r'}i(\cone(\sphere\cap \sphere')^{(0)})
	\end{align}
	for some $r'>0$. Applying (\ref{e:intersection}) to subcomplexes of $\sphere$, $\sphere'$ and $\sphere\cap \sphere'$ respectively, we deduce that
	\begin{align}
	\label{e:intersection2}
	&	i(\cone(\Delta)^{(0)})\cap_r i(\cone(\sphere\cap \sphere')^{(0)})\subset_{r'} i(\cone(\Delta\cap \sphere\cap \sphere')^{(0)}) \\
	\label{e:intersection3} & i(\cone(\Delta')^{(0)})\cap_r i(\cone(\sphere\cap \sphere')^{(0)})\subset_{r'} i(\cone(\Delta'\cap \sphere\cap \sphere')^{(0)}) \\
	\label{e:intersection4} & i(\cone(\Delta\cap \sphere\cap \sphere')^{(0)})\cap_r i(\cone(\Delta'\cap \sphere\cap \sphere')^{(0)}) \textmd{\ is\ finite}.
	\end{align}
Condition (\ref{e:diverge}) follows from (\ref{e:intersection1}), (\ref{e:intersection2}), (\ref{e:intersection3}) and (\ref{e:intersection4}).
\end{proof}

\begin{remark}
Theorem~\ref{t:OK} and its proof can be extended to the case of an arrangement of free abelian subgroups  $\ca$ isomorphic to the intersection of the integer lattice with some real subspace arrangement (not necessarily a coordinate subspace arrangement).
\end{remark}

\begin{remark}
The proof of Theorem \ref{t:OK} can be simplified simpler if we know the abelian subgroups are quasi-isometrically embedded. This is indeed the case for standard abelian subgroups in an Artin group (cf.\ Lemma \ref{l:qi}).
\end{remark}

\section{The action dimension of $A$}\label{s:acdim} 
First, consider the case when $A$ is  spherical.  The  action dimension of any braid group is computed in \cite{bkk}.  More generally, by using Theorem~\ref{t:OL}, Le \cite{le} computed the action dimension of any spherical Artin group.   Here is the result.

\begin{Theorem}\label{t:spherical}\textup{(cf.\ \cite{bkk}*{p.\,234}, \cite{le}*{Thm.~4.10}).} Suppose $A$ is a spherical Artin group associated to $(W,S)$.  Let $T_1,\dots , T_k$ be the irreducible components of $\bD(W,S)$ and put $d_i=\Card(T_i) -1$.  Then $\actdim A= k+\sum 2d_i$.  In other words, $\actdim A=2\gdim A - k$.
\end{Theorem}

\begin{proof}
The nerve of $(W,S)$ is a simplex $\gs$ of dimension $d=\Card(S)-1$. First consider the case $k=1$, i.e., the case where $(W,S)$ is irreducible.  The barycenter $v$ of $\gs$ is then a vertex of $\gs_\oslash$; the link of $v$ in $\gs_\oslash$ is $(\partial \gs)_\oslash$; and $\gs_\oslash$ is the cone, $(\partial \gs)_\oslash * v$.  Since  $(\partial \gs)_\oslash$ is a triangulation of $S^{d-1}$, it follows from Theorem~\ref{t:OL} that $O(\partial \gs_\oslash)$ is a $(2d-2)$-obstructor.  By \cite{bkk}*{Lemma 9}, $O\gs_\oslash=O(\partial \gs_\oslash) * S^0$ is a $(2d-1)$-obstructor. Therefore, $\obdim A\le 2d+1$.  Since the hyperplane complement in $S^{2d+1}$ is a $(2d+1)$-manifold, $\actdim A\ge 2d+1$.  By \eqref{e:actinequality}, both inequalities are equalities; so, $\actdim A=2d+1=\obdim A$.

Consider the general case, $A=A_{T_1}\times \cdots A_{T_k}$.  By \cite{bkk}*{Lemma 6}, for any two  groups $\pi_1$, $\pi_2$, we have $\actdim (\pi_1\times \pi_2)\le \actdim \pi_1 +\actdim \pi_2$ and 
$\obdim (\pi_1\times \pi_2)\ge \obdim \pi_1 +\obdim \pi_2$. Therefore, 
\[
\actdim A \le \sum_{i=1}^k (2d_i+1)\le \obdim A,
\]
which proves the lemma.
\end{proof}

Using Theorems~\ref{t:ados2} and \ref{t:OL} we can extend the results about $\raag$s in Example~\ref{ex:raags} to general Artin groups.  This gives the following restatement of the Main Theorem in the Introduction.

\begin{Theorem}\label{t:main}
Suppose $A$ is an Artin group associated to a Coxeter system with nerve $L$ of dimension $d$.  If $H_d(L;\zz/2)\neq 0)$, then $\obdim A\ge 2d+2$.  So, if the $K(\pi,1)$-Conjecture holds for $A$, then $\actdim A=2d+2$.
\end{Theorem}

\begin{proof}
By Theorems \ref{t:ados2} and \ref{t:OL}, $\obdim A\ge 2d+2$. If the $K(\pi,1)$-Conjecture holds for $A$, then, by Corollary~\ref{c:cd}, $\gdim A=d+1$ and consequently, the maximum possible value for its action dimension is $2d+2$.
\end{proof}
The other statement in Example~\ref{ex:raags} is that in the case of a $\raag$, if $H_d(L;\zz/2)=0$, then $\actdim A\le 2d+1$.  In her PhD thesis Giang Le \cite{le}  proved that the same conclusion holds for general Artin groups under the slightly weaker hypothesis that $H^d(L;\zz)=0$.

\begin{Theorem}\label{t:le}\textup{(Le \cite{le}).} 
Suppose $A$ is an Artin group associated to a Coxeter system with  $d$-dimensional nerve $L$.  Further suppose that the $K(\pi,1)$-Conjecture holds for $A$.  If $H^d(L;\zz)=0$ (and $d\neq 2$), then $\actdim A\le 2d+1$.
\end{Theorem}

\begin{giang}
As explained in Example~\ref{ex:raags}, when $A$ is a $\raag$, if $H_d(L;\zz/2)= 0$, then $OL$ is not a $2d$-obstructor and (provided $d\neq 2$), $OL$ embeds as a full subcomplex of some flag triangulation $J$ of $S^{2d}$ (cf.\ the proof of Proposition~2.2 in \cite{ados}).  Since the Davis complex for the right-angled Coxeter group $W_J$ is a contractible $(2d+1)$-manifold and since $A<W_{OL}<W_J$, we see that $\actdim A\le 2d+1$.  Similarly, for a general Artin group $A$, the vanishing of $H_d(L;\zz/2)$ implies that $OL_\oslash$ is a full subcomplex of some flag triangulation $S^{2d}$. However, we do not know how to deduce from this that $A$ acts properly on a contractible $(2d+1)$-manifold.  Le's argument in \cite{le} is different.  The condition $H^d(L;\zz)=0$ implies $H_d(L;\zz)=0$ and that $H_{d-1}(L;\zz)$ is torsion-free.  Standard arguments show that one can attach cells to $L$ to embed it into a contractible complex $L'$  of the same dimension $d$ (provided $d\neq 2$).  We can assume $L'$ is a simplicial complex. For each simplex $\gs$ of $L$, there is a model for $BA_\gs$ by a $(2d+1)$-manifold with boundary.   Le proves Theorem~\ref{t:le} by showing that one can glue together copies of these manifolds along their boundaries, in a fashion dictated by $L'$, to obtain another manifold with boundary $M$ with fundamental group $A$ so that when the $K(\pi,1)$-Conjecture holds for $A$, $M$ is a model for $BA$. 
\end{giang}

\begin{remark}
The Action Dimension Conjecture is the conjecture that if a group $\pi$ acts properly on a contractible manifold $M$, then the $\eltwo$-Betti numbers of $\pi$ vanish above the middle dimension of $M$.  In \cite{ados}*{\S7} it is explained  how the calculations for $\raag$s provide evidence for this conjecture.  The same discussion applies to the results in this paper for general Artin groups.  To wit, when the $K(\pi,1)$-Conjecture holds for an Artin group $A$, it is proved in \cite{dl} that the $\eltwo$-Betti number of $A$ in degree $(k+1)$ is equal to the ordinary (reduced) Betti number of $L$ (or $L_\oslash$) in degree $k$; in particular, the $\eltwo$-Betti numbers of $A$ vanish in degrees $>d+1$, when $d=\dim L$.
\end{remark}

\bibliographystyle{88}
\bibliography{1}

Michael W. Davis, Department of Mathematics, The Ohio State University, 231 W. 18th Ave., Columbus, Ohio 43210, \url{davis.12@osu.edu}

Jingyin Huang, Department of Mathematics and Statistics, McGill University, 805 Sherbrooke W., Montreal QC H3A 0B9, Canada, \url{jingyin.huang@mcgill.ca}

\obeylines
\end{document}